\documentclass[12pt,letterpaper]{amsart}
\usepackage{geometry}
\usepackage{hyperref}
\usepackage{mathpazo} 
\usepackage{url,setspace,xcolor,graphicx}
\newtheorem{theorem}{Theorem}
\newtheorem*{theorem*}{Theorem}
\newtheorem{proposition}{Proposition}
\newtheorem{lemma}{Lemma}
\newtheorem{corollary}{Corollary}
\theoremstyle{remark}
\newtheorem{remark}{Remark}
\newcommand{\C}{\mathbb{C}}

\newcommand{\D}{\Omega}
\newcommand{\ep}{\varepsilon}
\newcommand{\Dc}{\overline{\Omega}}
\newcommand{\dbar}{\overline{\partial}}

\onehalfspace

\title{Compactness of Hankel Operators with Continuous Symbols}
\author{Timothy G. Clos}
\author{S\"{o}nmez \c{S}ahuto\u{g}lu}
\email{Timothy.Clos@rockets.utoledo.edu; Sonmez.Sahutoglu@utoledo.edu}
\address{University of Toledo, Department of Mathematics \& Statistics, 
Toledo, OH 43606, USA}

\subjclass[2010]{Primary 47B35; Secondary 32W05}
\keywords{Hankel operator, Reinhardt, compact, convex}
\date{\today}
\begin{document}

\begin{abstract}
Let $\Omega$ be a bounded convex Reinhardt domain in $\mathbb{C}^2$ and 
$\phi\in C(\overline{\Omega})$. We show that the Hankel operator $H_{\phi}$ is 
compact if and only if $\phi$ is holomorphic along every non-trivial analytic 
disc in the boundary of $\Omega$.
\end{abstract}
\maketitle

Let $\D$ be a domain in $\C^n$ and  let $L^2(\D)$ and $A^2(\D)$ denote 
square integrable functions on $\D$ and  the Bergman space on $\D$ 
(the set of square integrable holomorphic functions on $\D$), respectively. 
Since $A^2(\D)$ is a closed subspace in $L^2(\D)$ the Bergman projection 
$P:L^2(\D)\to A^2(\D)$, the orthogonal projection, exists. Furthermore, 
let $H_{\phi}f=(I-P)(\phi f)$ for all $f\in A^2(\D)$ and $\phi\in L^{\infty}(\D)$.  
We note that $H_{\phi}$ is called the Hankel operator with symbol $\phi$. 
We refer the reader to \cite{PellerBook,ZhuBook} and references there in 
for more information on these operators.

Hankel operators form an active research area in operator theory. Our interest 
lies in their compactness properties in relation to the behavior of the symbols 
on the boundary of the domain. On the unit disc $\mathbb{D}$ in $\C$ Axler 
(\cite{Axler86}) showed that, for $f$ holomorphic on the unit disc $\mathbb{D}$,   
the Hankel operator $H_{\overline{f}}$ is compact on $A^2(\mathbb{D})$ if and 
only if  $f$ is in the little Bloch space (that is, $(1-|z|^2)|f'(z)|\to 0$ as $|z|\to 1$). 
This result has been extended into higher dimensions by Peloso (\cite{Peloso94}) in 
case the domain is smooth bounded and strongly pseudoconvex. The same year, 
Li (\cite{Li94}) characterized bounded and compact Hankel operators on 
strongly pseudoconvex domains for symbols that are square integrable only. 
Recently, \v{C}u\v{c}kovi\'c and the second author 
\cite[Theorem 3]{CuckovicSahutoglu09} gave a characterization for 
compactness of Hankel operators on smooth bounded convex 
domains in $\C^2$ with symbols smooth up to the boundary. 
We note that even though they stated their result for smooth 
domains and smooth symbols on the closure, examination of the 
proof shows that $C^1$-smoothness of the domain and the symbol 
is sufficient. They proved the following theorem.

\begin{theorem*}[\v{C}u\v{c}kovi\'c-\c{S}ahuto\u{g}lu]
Let $\D$ be a $C^1$-smooth bounded convex domain in $\C^2$ and $\phi\in C^1(\Dc)$. 
Then  the Hankel operator $H_{\phi}$ is compact on $A^2(\D)$ if and only if  
$\phi\circ f$ is holomorphic for any holomorphic function $f:\mathbb{D}\rightarrow b\D$.
\end{theorem*}

In this paper we prove a similar result with symbols that are only continuous up to 
the boundary. 	The first result in this direction was proven by Le in \cite{Le10}. 
He showed that for $\D=\mathbb{D}^n$, the polydisc in $\C^n$, and 
$\phi\in C(\Dc)$, the Hankel operator $H_{\phi}$ is compact on $A^2(\D)$ 
if and only if $\phi=f+g$ where $f$ and $g$ are continuous on $\Dc,f=0$ 
on $b\D$, and $g$ is holomorphic on $\D$. 
We prove the following theorem, generalizing Le's result in $\C^2$.  

\begin{theorem}\label{Thm1} 
Let $\Omega$ be a bounded convex Reinhardt domain in $\mathbb{C}^2$ and 
$\phi\in C(\overline{\Omega})$.  Then the Hankel operator $H_{\phi}$ is 
compact on $A^2(\Omega)$ if and only if $\phi\circ f$ is 
holomorphic for any holomorphic function $f:\mathbb{D}\rightarrow b\D$.
\end{theorem}

We note that in the theorem above there is no regularity restriction 
on the domain, but the class of domains is smaller than the one considered 
in \cite{CuckovicSahutoglu09}. It would be interesting to know if the same 
result is still true on convex domains in $\C^n$. 

\section*{Proof of Theorem \ref{Thm1}}
 Let us start by some notation. We denote 
 \[\mathbb{D}_r=\{z\in \mathbb{C}:|z|<r\}, S_r=\{z\in 
 \mathbb{C}: |z|=r\}, 
 A(0,\delta_1,\delta_2)=\{z\in \mathbb{C}: \delta_1<|z|<\delta_2\}\] 
 for $r, \delta_1,\delta_2>0$.  
 
 In the next lemma we prove that any analytic disc $\Delta_0\subset b\Omega$ 
 is contained in a disc that intersects the coordinate axis.  This allows us to simplify 
 the problem for convex Reinhardt domains, since any 
 disc in $b\Omega$ must be horizontal or vertical.

\begin{lemma}\label{Lem1}
Let $\D$ be a bounded convex Reinhardt domain in $\C^2$ and 
$\Delta\subset b\D$ be a non-trivial analytic disc. Then there exists 
 $r>0$ and $p\in \C$ such that either 
$\Delta\subset \mathbb{D}_r\times \{p\} \subset b\D$ or 
$\Delta\subset \{p\} \times \mathbb{D}_r\subset b\D$.
\end{lemma}

\begin{proof} 
	Suppose that $F(\mathbb{D})=\Delta$ is a non-trivial 
	disc in $b\D$ where $F(\xi)=(f(\xi),g(\xi))$.  Then either $f'g'\equiv 0$ or 
	there exists $\xi_0\in \mathbb{D}$ such that $f'(\xi_0)g'(\xi_0)\neq 0$.
	In case $f'g'\equiv 0$, by identity principle, we conclude that either 
	$f'\equiv 0$ or $g'\equiv 0$. That is, either $f$  or $g$ is constant. 
	
	On the other hand, if  $f'(\xi_0)g'(\xi_0)\neq 0$ then the disc $\Delta$ 
	is a smooth complex curve in a neighborhood $F(\xi_0)$. Furthermore, 
	the fact that $\D$ is Reinhardt domain in $\C^2$ implies  
	that $b\D$ is smooth locally in a neighborhood of $F(\xi_0)$. 
	This can be seen as follows: Without loss of generality we assume 
that $f(\xi_0)\neq 0$. Let $\xi_0=x_0+iy_0$ and 
\[G(x,y,\theta)=(e^{i\theta}f(x+iy),g(x+iy)).\] 
Then one can show that 	the image of $G$ is a smooth surface in $\C^2$ 
near $G(\xi_0, 0)=F(\xi_0)$ as the Jacobian of $G$ is of rank 3 at $(\xi_0,0)$. 
Since $b\D$ is  a 3 dimensional surface  we conclude that the boundary of 
$\D$ is smooth near near $F(\xi_0)$ as it can be seen as the image 
of $G(x,y,\theta)$. 
Then  we can apply \cite[Lemma 2]{CuckovicSahutoglu14}  
(since $b\D$ is smooth near $F(\xi_0)$) and use the identity principle 
to conclude that  either $f$ or $g$ is constant.  We reach a  contradiction with the 
	assumption that $f(\xi_0)\neq 0$. Therefore, either $\Delta$ is flat 
	and horizontal ($g$ is constant) or flat and vertical ($f$ is constant).  
	
	For the rest of the proof, without loss of generality, we assume that 
	$\Delta$ is horizontal.  There exists $p\in \mathbb{C}$,  $\delta_1> 0$,  
	and $\delta_2>0$ such that 
\[A(0,\delta_1,\delta_2)\times \{p\}\subset b\Omega.\]  
The assumption that $\Omega$ is convex and Reinhardt implies that 
$\Omega$ is complete. So, 
\begin{align}\label{Eqn0}
\{(z,w)\in \mathbb{C}^2: |z|\leq \delta_2, |w|\leq |p|\}\subset \overline{\Omega}.
\end{align}
Next, we will show that 
$\{(z,w)\in \mathbb{C}^2:|z|\leq \delta_1, |w|>|p|\}\cap \Omega=\emptyset$.  
Suppose that there exists 
$(z_0,w_0)\in\{(z,w)\in \mathbb{C}:|z|\leq \delta_1, |w|>|p|\}\cap\Omega$ 
and let $z\in \mathbb{C}$ such that $|z|=\delta_2$. 
We choose $\lambda>0$ small 
enough such that $(|z|-\lambda,|p|-\lambda)\in \Omega$ and the line segment 
joining $(|z|-\lambda,|p|-\lambda)$ with $(z_0,w_0)$, called $L_1$, is such that 
\[L_1\cap (A(0,\delta_1,\delta_2)\times \{|p|e^{i\theta}:0\leq \theta\leq 2\pi \})\neq \emptyset.\]  
However, since $\overline{A(0,\delta_1,\delta_2)}\times \{|p|e^{i\theta}:0\leq \theta\leq 2\pi 
\}\subset b\Omega$, we conclude $L_1\cap b\Omega\neq \emptyset$.  
Since the initial and terminal points of $L_1$ lie in $\Omega$ and $\Omega$ 
is convex, we arrive at a contradiction.  This shows that 
$\{(z,w)\in \mathbb{C}^2:|z|\leq \delta_1, |w|>|p|\}\cap \Omega=\emptyset$. Combining 
this with \eqref{Eqn0} we conclude that 
$\{(z,w)\in \mathbb{C}^2: |z|\leq \delta_2, |w|=|p|\}\subset b\D.$
\end{proof}
We take this opportunity to correct a typo in 
\cite[Lemma 2]{CuckovicSahutoglu14}. In the statement of the lemma, the word 
``complete'' should be ``convex''.  The lemma is proven for the correct 
domains: piecewise smooth bounded convex Reinhardt domains in $\C^2$. 
\begin{remark}
	Lemma \ref{Lem1} implies that if $\Omega\subset\mathbb{C}^2$ is a bounded convex 
	Reinhardt domain with a piecewise smooth boundary, then any horizontal analytic disc in 
	$b\Omega$ is contained in $\mathbb{D}_r\times S_{q}$ for some $r>0$ and $q>0$. Likewise,  
	any vertical analytic disc in $b\D$ is contained in $S_{q'}\times \mathbb{D}_{r'}$ for some 
	$r'>0$ and $q'>0$. 
\end{remark}   

	\begin{figure}[h]
		\centering
		\includegraphics[width=0.2\linewidth]{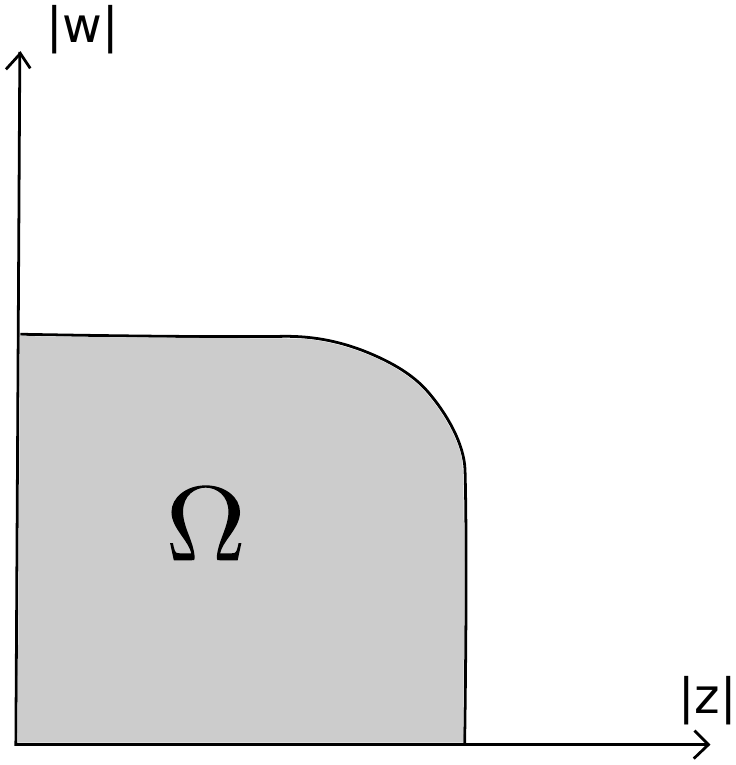}
		\label{GraphConvexReinhardt}
	\end{figure}

 As in \cite{CuckovicSahutoglu14} we represent a complete Reinhardt domain 
 $\D\subset \C^2$ as union of  horizontal slices. In other words, let 
$H_{\D}$ be an open disc in $\C$ such that 
 \begin{align}\label{EqnSlice}
 \Omega=\bigcup_{w\in H_{\D}}\Delta_w\times \{w\}
 \end{align}
where $\Delta_w=\{z\in \mathbb{C}:|z|<r_w\}$ is the slice of $\D$ at $w$ level. 
That is, $(z,w)\in \D$ if and only if $|z|<r_w$.  

\begin{lemma}[\cite{CuckovicSahutoglu14}]\label{Lem2}
Let $\phi\in C(\mathbb{C})$ and $f:\mathbb{C}\rightarrow \mathbb{C}$ be an 
entire function.  Then 
\[\|H_{\phi}^{\mathbb{D}_r} f\|_{L^2(\mathbb{D}_r)}\rightarrow 
\|H_{\phi}^{\mathbb{D}_{r_0}} f\|_{L^2(\mathbb{D}_{r_0})}\] as 
$r\rightarrow r_0$.
\end{lemma}

Lemma \ref{Lem1}, Lemma \ref{Lem2}, and  \cite[Lemma 3]{CuckovicSahutoglu14} 
imply the following corollary.

\begin{corollary}\label{Cor1}
	Let $\D$ be a bounded convex Reinhardt domain in 
	$\C^2, \phi\in C(\mathbb{C})$, and $\Delta_{w_0}\times \{w_0\}$ 
	be a non-trivial analytic disc in $b\D$ where $w_0\in bH_{\D}$.  Then 
\[\lim_{H_{\D}\ni w\rightarrow w_0}\|H^{\Delta_w}_{\phi}(1)\|_{L^2(\Delta_w)}
=\|H^{\Delta_{w_0}}_{\phi}(1)\|_{L^2(\Delta_{w_0})}.\]
\end{corollary}

\begin{lemma}\label{Lem3} 
Let $\Omega$ be a bounded convex Reinhardt domain in $\C^2$ and 
$\phi\in C(\overline{\Omega})$. Furthermore, let $w_0\in bH_{\D}$ and 
$\phi_0(z,w)=\phi(z,w_0)$.  Assume that $H_{\phi}$ is compact on 
$A^2(\Omega)$ and  $\{g_j\}$ is a bounded sequence in $A^2(H_{\D})$ 
such that $g_j\rightarrow 0$ uniformly on $H_{\D}\setminus V$ as 
$j\to\infty$ for any open set $V$ containing $w_0$. Then 
$H_{\phi_0}g_j\rightarrow 0$ as $j\rightarrow \infty$.     
\end{lemma}

\begin{proof}
We note that $g_j\rightarrow 0$ weakly in $A^2(\Omega)$ as $j\rightarrow 0$.  
Hence, by compactness of $H_{\phi}$ we have 
$\|H_{\phi}g_j\|_{L^2(\Omega)}\rightarrow 0$ as $j\rightarrow 
\infty$.  Now, we write \[\|H_{\phi_0}g_j\|_{L^2(\Omega)}\leq 
\|H_{\phi-\phi_0}g_j\|_{L^2(\Omega)}+\|H_{\phi}g_j\|_{L^2(\Omega)}.\]  
So, we just consider the first term on the right hand side of the above inequality. 
Since $\{g_j\}$ is a bounded sequence, there exists $M>0$ such that 
$\|g_j\|^2_{L^2(\Omega)}\leq M$ for all $j\in \mathbb{N}$.  
Furthermore, since $\phi-\phi_0$ is continuous on $\overline{\Omega}$ 
and $\phi-\phi_0=0$ on $\overline{\Delta_{w_0}}$ for all $\varepsilon>0$, 
there exists $\delta>0$ such that  
\[\sup\{|\phi(z,w)-\phi_0(z,w)|^2:(z,w)\in \overline{\Omega}, 
|w-w_0|\leq\delta\}<\frac{\varepsilon}{2M}.\]  
We note that, below, $V(\D)$ denotes  the volume of $\D$ with respect 
to Lebesgue measure. 
\begin{align*}
\|H_{\phi-\phi_0}g_j\|^2_{L^2(\Omega)}
\leq&\|({\phi-\phi_0})g_j\|^2_{L^{2}(\{(z,w)\in \Omega: |w-w_0|\leq \delta\})}\\
&+V(\D)\|({\phi-\phi_0})g_j\|^2_{L^{\infty}(\{(z,w)\in \Omega: |w-w_0|>\delta\})}\\
<&\frac{\ep}{2} +V(\Omega)\|({\phi-\phi_0})g_j\|^2_{L^{\infty}(\{(z,w)\in \Omega: 
|w-w_0|>\delta\})}.
\end{align*}
Since $(\phi-\phi_0)\in C(\Dc)$ and $g_j\rightarrow 0$ uniformly on 
$\{(z,w)\in \D: |w-w_0|>\delta\}$ as $j\to\infty$, we conclude that 
for any $\delta,\varepsilon>0$ there exists 
$j_0\in \mathbb{N}$ such that 
\[V(\Omega)\|({\phi-\phi_0})g_j\|^2_{L^{\infty}(\{(z,w)\in \Omega: 
|w-w_0|>\delta\})}<\frac{\varepsilon}{2}\]
for $j\geq j_0$. Therefore, 
\[\|H_{\phi-\phi_0}g_j\|^2_{L^2(\Omega)}<\varepsilon\]
 for $j\geq j_0$ and  the proof of the lemma is complete.
\end{proof}

Before we state the next lemma some explanation about the notation 
is in order. We think of the operators as defined on spaces on $\D$ 
unless the domain is indicated as a superscript. For instance, for an 
open subset $V$ of $\D$ the operators $H_{\phi}^{V}$ and $P^{V}$ 
are  defined on $A^2(V)$ and $L^2(V)$, respectively; 
whereas, $H_{\phi}$ and $P$ are defined on $A^2(\D)$ and $L^2(\D)$, 
respectively. Furthermore, in the next two lemmas, 
we think of $\phi$ as a function of $z$ (or as a function of $(z,w)$ but 
independent of $w$). For instance, $\phi$ is a function of $z$ in $H^{\Delta_w}_{\phi}$ 
and a function $(z,w)$ (but independent of $w$) in $H_{\phi}$.

The following lemma is a special case of equation (3) in 
\cite[pg. 637]{CuckovicSahutoglu14} for 	
$\phi=\psi_0=\phi_0$ and $f_1= f_2\equiv 1$. 
 
\begin{lemma}[\cite{CuckovicSahutoglu14}]\label{Lem4}
	Let $\D$ be a bounded convex Reinhardt domain in $\C^2$ and 
	$\phi \in C(\Dc)$ such that $\phi(z,w)=\phi(z,0)$ for $(z,w)\in \D$.
	Then 
		\begin{align*}
		\| H_{\phi}g\|^2_{L^2(\Omega)}
		=&\int_{H_{\D}}|g(w)|^2\int_{\Delta_w}|H^{\Delta_w}_{\phi}(1)(z)|^2dV(z)dV(w)\\
		&+\int_{\Omega}(H_{\phi}g)(z,w)\overline{P^{\Delta_w}(\phi)(z)g(w)}dV(z,w)
		\end{align*}
		for  $g\in A^2(H_{\D})$
\end{lemma}
Similarly the  following lemma is included in \cite[pg 640]{CuckovicSahutoglu14} 
again  for 	$\phi=\psi_0=\phi_0$ and $f_1= f_2\equiv 1$. 

\begin{lemma}[\cite{CuckovicSahutoglu14}]\label{Lem5} 
Let $\D$ be a bounded convex Reinhardt domain in $\C^2$ and 
$\phi \in C(\Dc)$ such that $\phi(z,w)=\phi(z,0)$ for $(z,w)\in \D$. Assume 
that $\{g_j\}$ is a bounded sequence in $A^2(H_{\D})$ such that $g_j\rightarrow 0$ 
uniformly on $\overline{H}_{\D}\setminus V$ for any open set $V$ containing 
$w_0\in H_{\D}$. Then 
\[\int_{\Omega}(H_{\phi}g_j)(z,w)\overline{P^{\Delta_w}(\phi)(z)g_j(w)}dV(w,z)
\rightarrow  0 \text{ as }j\rightarrow \infty.\]
\end{lemma}

The next lemma allows us to approximate the symbol with smooth appropriate 
symbols. We define $\Gamma_{\D}\subset b\Omega$ to be the closure of the 
union of all non-trivial analytic discs in $b\Omega$. That is, 
\begin{equation}\label{Equation 2}
	\Gamma_{\D}=\overline{\bigcup\{f(\mathbb{D}):f:\mathbb{D}\to b\D 
		\text{ is non-constant holomorphic mapping}\}}.
\end{equation}
 
\begin{lemma}\label{Lem6} 
Let $\Omega$ be a bounded convex Reinhardt domain in $\C^2$ that is not the 
product of two discs.  Assume that $\Gamma_{\D}\neq \emptyset$ and  
$\phi\in C(\overline{\Omega})$ such that $\phi\circ f$ is holomorphic for 
any holomorphic function $f:\mathbb{D}\rightarrow b\Omega$. Then 
there exists $\{\psi_n\}\subset C^{\infty}(\overline{\Omega}) $ such that 
\begin{itemize}
	\item[i.] $\psi_n\circ f$ is holomorphic for all $n$ and for any holomorphic 
	function $f:\mathbb{D}\rightarrow b\Omega$,
	\item[ii.] $\|\psi_n-\phi\|_{L^{\infty}(\Gamma_{\D})}\rightarrow 0$ as $n\to\infty$.
\end{itemize} 
\end{lemma}

\begin{proof}
Let $\Delta_1=\mathbb{D}_{r_1}\times S_{s_1}$ be the family of horizontal 
analytic discs in $b\Omega$ as outlined in Lemma \ref{Lem1}.  
Then for $0<r<1$ we define 
\[\phi_{r}(z,w)=\phi(rz,w).\]  
Since $\phi\in C(\overline{\Omega})$, one can show that 
\[\phi_r\rightarrow \phi \text{ uniformly on } \Dc \text{ as } r\rightarrow 1^-.\]
We consider $\phi$, restricted to $\overline{\Delta_1}$, 
to be a function of $(z,\theta)$ for $z\in \overline{\mathbb{D}}_{r_1}$ and periodic in 
$\theta\in\mathbb{R}$ with period $2\pi$.  By assumption, the function 
$\phi_r(.,\theta)$ is holomorphic 
on a neighborhood of $\overline{\mathbb{D}}_{r_1}$ for every 
$\theta \in \mathbb{R}$.  Let $\gamma\in C^{\infty}_0( (-1,1))$  be such that 
 $\gamma\geq 0$ and $\int_{-1}^{1}\gamma(\theta)d\theta=1$.  Similarly, 
 let $\chi\in C^{\infty}_0(\mathbb{D}_{r_1})$  be such that $\chi\geq 0$ and 
 $\int_{\mathbb{D}_{r_1}}\chi(z)dV(z)=1$.   Now, we define 
$\gamma_{\delta}(\theta)=\delta^{-1}\gamma(\theta/\delta)$ and 
$\chi_{\varepsilon}(z)=\varepsilon^{-2}\chi(z/\varepsilon)$.  Notice that 
$\{\gamma_{\delta}\}_{\delta>0}$ and $\{\chi_{\ep}\}_{\ep>0}$ are approximate 
identities.  We define the convolution 
\[C_{r,\ep}^{\phi}(z,\theta)
=\int_{-\pi}^{\pi}\int_{\mathbb{D}_{r_1}}\phi(r(z-\alpha),(\theta-\theta'))
\chi_{\ep}(\alpha)\gamma_{\ep}(\theta')dV(\alpha)d\theta'.\]
One can show that  for $\varepsilon>0$ sufficiently small (depending on $r$) 
the function  $C_{r,\ep}^{\phi}(.,\theta)$ is holomorphic on a neighborhood of  
$\overline{\mathbb{D}}_{r_1}$ for every $\theta\in \mathbb{R}$.  
Also the assumption that $\phi\in C(\overline{\Omega})$ implies that 
\[ C_{r,\ep}^{\phi}\rightarrow \phi_r \text{ uniformly on }\overline{\Delta_1} 
\text{ as } \ep\to 0^+\]
for all $0<r<1$.  Therefore, the functions $C_{r,\ep}^{\phi}$ are holomorphic 
``along" horizontal analytic discs in $b\Omega$ for small $\ep>0$.  Now, 
we extend $C_{r,\ep}^{\phi}$ as a $C^{\infty}$-smooth function onto  
$\overline{\Omega}$ and call this extension $\widetilde{C}_{r,\ep}^{\phi}$. 

If $b\D$ contains non-trivial vertical analytic discs $\Delta_2$ then we can 
use  a similar construction on  $\Delta_2$.  That is, using the regularization 
procedure outlined above in this proof, we can construct a collection of functions 
$\widetilde{B}_{r,\ep}^{\phi}\in  C^{\infty}(\overline{\Omega})$
 such that $\widetilde{B}_{r,\ep}^{\phi}$ are holomorphic ``along" any 
 vertical analytic disc in $\Delta_2$ for small $\ep>0$ and 
 \[\widetilde{B}_{r,\ep}^{\phi}\rightarrow \phi_r \text{ uniformly 
 on } \overline{\Delta}_2 \text{ as } \ep\to 0^+\] 
 for all $0<r<1$.  Since $\Omega$ is not the product of discs, 
  (hence $\overline{\Delta_1}\cap \overline{\Delta_2}=\emptyset$), 
there exists open sets $F$ and $G$ such that $\overline{\Delta_1}\subset F$, 
$\overline{\Delta_2}\subset G$, and $\overline{F}\cap \overline{G}=\emptyset$.  
Then we choose $\chi_F,\chi_G\in C^{\infty}_0(\mathbb{C}^2)$ 
such that $0\leq \chi_G,\chi_F\leq 1$, $\chi_G\equiv 1$ on $G, \chi_F\equiv 1$ on $F$, 
and $\chi_F+\chi_G\equiv 1$ on $\overline{\Omega}$.  

We define 
\begin{equation}\label{Equation 1}
\phi^{r,\ep}=\chi_F \widetilde{C}_{r,\ep}^{\phi}+\chi_G \widetilde{B}_{r,\ep}^{\phi}.
\end{equation}  
By construction, $\chi_F\equiv 0$ on $G$ and $\chi_G\equiv 0$ on 
$F$. Furthermore, $\widetilde{C}_{r,\ep}^{\phi}$ is holomorphic along $\Delta_1$, 
and $\widetilde{B}_{r}^{\phi}$ is holomorphic along $\Delta_2$ for small $\ep>0$.  
For $n=1,2,\ldots $ we choose $r_n=(n-1)/n$ and $\ep_n\to 0^+$ so that 
\begin{itemize}
	\item[i.]  $\phi^{r_n,\ep_n}\circ h$ is holomorphic for all $n$ and every 
	holomorphic $h:\mathbb{D}\to b\D$,
	\item[ii.] $\phi^{r_n,\ep_n}\to \phi $ uniformly on $\Gamma_{\D}$ 
	as $n\to \infty$.
\end{itemize}
Finally, we finish the proof by defining $\psi_n=\phi^{r_n,\ep_n}$.
\end{proof}

Let $X$ and $Y$ be two normed linear spaces and $T:X\to Y$ be a  bounded 
linear operator. We define the essential norm of $T$, denoted by  $\|T\|_e$, as  
\[\|T\|_e=\inf\{\|T-K\|:K:X\to Y \text{ is a compact operator}\}\]
where $\|.\|$ denotes the operator norm. 

\begin{lemma}\label{Lem7} 
	Let $\Omega$ be a bounded convex domain in $\C^n$ and  
	$\Gamma_{\Omega}\neq \emptyset$  be defined as in \eqref{Equation 2}. 
	Assume that $\{\phi_n\}\subset C(\Dc)$ is a sequence such that 
	$\phi_n\to 0$ uniformly on $\Gamma_{\D}$ as $n\to\infty$. Then  
	$\lim_{n\to \infty}\|H_{\phi_n}\|_e= 0$.
\end{lemma}

\begin{proof}
Let $\ep>0$. Then there exists $N$ such that 
$\sup\{|\phi_n(z,w)|:(z,w)\in \Gamma_{\D}\}<\ep$ for $n\geq N$. 
For $n\geq N$ we choose an open neighborhood $U_{n,\ep}$ of 
$\Gamma_{\D}$ such that $|\phi_n(z,w)|<\ep$ for $(z,w)\in U_{n,\ep}$.  
Furthermore, we choose a smooth cut-off function 
$\chi_{n,\ep}\in C^{\infty}_0(U_{n,\ep})$ such that  $0\leq \chi_{n,\ep}\leq 1$ 
and $\chi_{n,\ep}=1$ on a neighborhood of $\Gamma_{\D}$.

Let us define 
\[\phi_{1,n,\ep}=\chi_{n,\ep}\phi_n \text{ and } \phi_{2,n,\ep}=(1-\chi_{n,\ep})\phi_n.\] 
Then $\phi_{n}=\phi_{1,n,\ep}+\phi_{2,n,\ep}$ and $|\phi_{1,n,\ep}|< \ep$ on 
$\Dc$ while $\phi_{2,n,\ep}=0$ on a neighborhood of $\Gamma_{\D}$ in $\Dc$. 
Furthermore, 
\[\|H_{\phi_{1,n,\ep}}\|_e\leq\|H_{\phi_{1,n,\ep}}\|
\leq \sup\{ |\phi_{1,n,\ep}(z,w)|:(z,w)\in \Dc\}<\ep.\]

Next we will show that  $H_{\phi_{2,n,\ep}}$ is compact. Since $\phi_{2,n,\ep}=0$ 
on a neighborhood of $\Gamma_{\D}$ in $\Dc$, using convolution with 
approximate  identity, one can choose  $\{\psi_{k,n,\ep}\}\subset  C^{\infty}(\Dc)$ 
such that $\psi_{k,n,\ep}=0$ on a neighborhood of $\Gamma_{\D}$ in $\Dc$ 
for all $k$ and $\psi_{k,n,\ep}\to\phi_{2,n,\ep}$ uniformly on $\Dc$ as 
$k\to\infty$.  We choose finitely many open balls $U_j=B(p_j,r_j)$ 
for $j=1,\ldots, N$ such that 
$\Gamma_{\D}\subset \cup_{j=1}^NU_j, p_j\in \Gamma_{\D}$ , and 
$\psi_{k,n,\ep}=0$ on $U_j$ for all $j$. Then we cover 
$b\D\setminus  \cup_{j=1}^NU_j$ by finitely many open balls 
$U_j=B(p_j,r_j)$ for $j=N+1,\ldots,M$ such that  $p_j\in b\D$ and 
$U_j\cap \Gamma_{\D}=\emptyset$ for $j=N+1,\ldots, M$. 

Below $R_V$ denotes the restriction operator onto $V\subset \D$. That is, 
$R_Vf=f|_V$ for $f\in A^2(\D)$. We note that $U_j\cap \D$ is a bounded 
convex domain with no analytic disc in the boundary for all $j=N+1,\ldots, M$. 
Then \cite[Theorem 1.1]{FuStraube98} (see also \cite[Theorem 4.26]{StraubeBook}) 
implies that the $\dbar$-Neumann operator on $U_j\cap \D$ is compact 
(for $j=N+1,\ldots,M$) and \cite[Proposition 4.1]{StraubeBook}, in turn,  
implies that the Hankel operator 
$H^{U_j\cap \D}_{R_{U_j\cap\D}(\psi_{k,n,\ep})}R_{U_j\cap\D}$ is compact 
for $j=N+1,\ldots, M$. 

Therefore, we have chosen finitely many balls $U_j=B(p_j, r_j)$ for 
$j=1,\ldots,M$ such that 
\begin{itemize}
	\item[i.] $p_j\in b\D$ and $b\D\subset \cup_{j=1}^M U_j$,
	\item[ii.] the operator $H^{U_j\cap \D}_{R_{U_j\cap\D}(\psi_{k,n,\ep})}R_{U_j\cap\D}=0$ 
	for $p_j\in \Gamma_{\D}$,   
	\item[iii.]  the operator $H^{U_j\cap \D}_{R_{U_j\cap\D}(\psi_{k,n,\ep})}R_{U_j\cap\D}$ 
	is compact for $p_j\not\in \Gamma_{\D}$. 
\end{itemize}	
So, the local Hankel operators 
$H^{U_j\cap \D}_{R_{U_j\cap\D}(\psi_{k,n,\ep})}R_{U_j\cap\D}$ are compact for all $j=1,\ldots,M$.  
Now we use \cite[Proposition 1, (ii)]{CuckovicSahutoglu09} to conclude that 
$H_{\psi_{k,n,\ep}}$ is compact. 
Hence $H_{\phi_{2,n,\ep}}$ is compact and $\|H_{\phi_n}\|_e\leq \ep$ for $n\geq N$. 
Therefore,  $\lim_{n\to \infty}\|H_{\phi_n}\|_e= 0$. 
\end{proof}

We will now show one implication of the main theorem if the symbol is smooth up to the 
boundary.  

\begin{lemma}\label{Lem8} 
Let $\Omega\subset \mathbb{C}^2$ be a bounded convex Reinhardt domain 
that is not the product of two discs and $\phi\in C^{\infty}(\overline{\Omega})$. 
Assume that  $\phi\circ f$ is holomorphic for any holomorphic function 
$f:\mathbb{D}\rightarrow b\Omega$. Then $H_{\phi}$ is compact on $A^2(\Omega)$.
\end{lemma}
\begin{proof}
If $b\D$ does not contain any non-trivial analytic disc the $\dbar$-Neumann 
operator is compact \cite[Theorem 4.26]{StraubeBook} (see also 
\cite[Theorem 1.1]{FuStraube98}). Furthermore, if the $\dbar$-Neumann 
operator is compact then $H_{\phi}$ is compact for all $\phi\in C(\Dc)$ 
\cite[Proposition 4.1]{StraubeBook}.  So if $b\D$ does not contain any 
non-trivial analytic disc, there is nothing to prove as the operator 
$H_{\phi}$ is compact. Lemma \ref{Lem1} implies that the analytic 
discs in $b\Omega$ are flat and horizontal or flat and vertical. We 
assume that there are non-trivial vertical and horizontal analytic 
discs in $b\D$ as the proof is even simpler if there are no vertical 
or horizontal discs. Let $\Delta_1$ and $\Delta_2$ be the horizontal 
and the vertical discs in $b\D$. So there exists $0<r_1<s_2,0<r_2<s_1$ 
(since $\D$ is not product of two discs) such that 
  \[\Delta_1=\mathbb{D}_{r_1}\times S_{s_1} \text{ and } 
  \Delta_2=S_{s_2}\times \mathbb{D}_{r_2}.\]
We note that $\Gamma_{\D}=\overline{\Delta_1}\cup\overline{\Delta_2}$ and 
$\overline{\Delta_1}\cap\overline{\Delta_2}=\emptyset$. Let us define
\[\phi_1(z,w)=\phi(z,w)-(|w|^2-s_1^2)\frac{1}{w}
\frac{\partial \phi(z,w)}{\partial \overline{w}}\]
for $w\neq 0$. We note that $\phi_1$ is a $C^{\infty}$-smooth function 
on $\Dc$ for $w\neq 0$ and $\phi_1=\phi$ on $\Delta_1$. Furthermore, 
using the fact that $\phi(.,w)$ is holomorphic on $\mathbb{D}_{r_1}$ 
for $|w|=s_1$, one can verify that $\dbar\phi_1=0$ on $\Delta_1$. 
Similarly we define 
\[\phi_2(z,w)=\phi(z,w)-(|z|^2-s_2^2)\frac{1}{z}
\frac{\partial \phi(z,w)}{\partial \overline{z}}\]
and one can verify that $\phi_2=\phi$ and $\dbar\phi_2=0$ on $\Delta_2$. 

We choose $\chi_1, \chi_2\in C^{\infty}(\overline{\Omega})$ such that 
\begin{itemize}
	\item [i.] $\chi_1\equiv 1$ on a neighborhood of $\overline{\Delta_1}$ and  
	$\chi_1\equiv 0$ on a neighborhood of $\overline{\Delta_2}\cup \{(z,w)\in \Dc:|w|=0\}$,
	\item [ii.] $\chi_2\equiv 1$ on a neighborhood of $\overline{\Delta_2}$ and  
		$\chi_2\equiv 0$ on a neighborhood of $\overline{\Delta_1}\cup \{(z,w)\in \Dc:|z|=0\}$.
\end{itemize}
Then we define 
 \[\psi=\chi_1 \phi_1+\chi_2 \phi_2\in C^{\infty}(\Dc).\]
We note that $\psi=\phi$  and $\dbar\psi=0$  on $\Gamma_{\D}$. 
Lemma \ref{Lem7} implies that $H_{\phi-\psi}$ is compact on $A^2(\D)$. 
To finish the proof we only need to show that  $H_{\psi}$ is compact. 
This can be done exactly in the same manner as the proof of 
$H^{\D}_{\widetilde{\beta}}$ is compact in \cite[pp 3740]{CuckovicSahutoglu09}.
\end{proof}

\begin{proposition}\label{PropBidisc}
Let $f\in C(\overline{\mathbb{D}^2})$ such that $f(e^{i\theta},.)$ and 
$f(.,e^{i\theta})$ are holomorphic on $\mathbb{D}$ for each fixed $\theta$. 
Then $H_f$ is compact on $A^2(\mathbb{D}^2)$.
\end{proposition}
\begin{proof}
Let $\mathbb{T}^2=\{(z,w)\in \mathbb{C}^2:|z|=|w|=1\}$ be the distinguished 
boundary and 
\[F_N(z,w)=\sum_{|m|,|j|\leq N}\left(1-\frac{|m|}{N+1}\right)
\left(1-\frac{|j|}{N+1}\right)a_{mj}(f)z^{m}w^{j}\]  
where 
\[a_{mj}(f)=\int_{\mathbb{T}^2}f(\zeta_1,\zeta_2)\zeta_1^{-m}\zeta_2^{-j}d\sigma(\zeta)\]
and $\sigma$ is the normalized Lebesgue measure on $\mathbb{T}^2$.  We let $S_{N,2}$ be 
the $N$-th Fej\'er kernel on $\mathbb{T}^2$.  As in  \cite[Chapter I, Section 9]{Katz04}, 
it is just the product of the $N$-th Fej\'er kernels on the circle. Since 
$f\in C({\mathbb{T}^2})$, and the convolution $S_{N,2}*f=F_N$, Fej\'er's Theorem 
on Ces\`aro summability (see, for example, \cite[Section 9.2, pg 64]{Katz04} for 
homogeneous Banach spaces) implies that  
\[\|F_N-f\|_{L^\infty(\mathbb{T}^2)}\rightarrow 0\] 
as $N\rightarrow \infty$.

Now we claim that $a_{mj}(P)=0$ for any holomorphic polynomial $P$ and 
$m\leq -1$ or $j\leq -1$.  Let 
\[P(z,w)=\sum_{l,k=0}^{n}b_{lk}z^lw^k\] 
and  $m\leq -1$ or $j\leq -1$. Then  
\begin{align}
\nonumber a_{mj}(P)=&\sum_{l,k=0}^n b_{lk}\langle\zeta_1^l\zeta_2^k,{\zeta_1}^m{\zeta_2}^j
\rangle_{L^2(\mathbb{T}^2)}\\
\label{Eqn1}=&\frac{1}{(2\pi)^2}\int_{0}^{2\pi}\int_{0}^{2\pi}\sum_{l,k=0}^{n}b_{lk}
e^{i\theta_1l}e^{i\theta_2k}e^{-i\theta_1m}e^{-i\theta_2 j}d\theta_1 d\theta_2\\
\nonumber =&0.
\end{align}  

Next we will show that $a_{mj}(f)=0$ for $m\leq -1$ or $j\leq -1$. Without loss 
of generality, we suppose that $j\leq -1$.  Since $f(e^{i \theta_1},.)$ is 
holomorphic on $\mathbb{D}$, using Mergelyan's Theorem, there 
exists a sequence of holomorphic polynomials 
$\{P_{n,\theta_1}\}_{n\in \mathbb{N}}$ converging  to $f$ 
uniformly on $\overline{\mathbb{D}}$ as $n\rightarrow \infty$. 
Let us define $P_{n,\theta_1,r}(\xi)=P_{n,\theta_1}( r\xi)$ and 
$f_r(z,w)=f(z,rw)$ for $0<r<1$. Then $P_{n,\theta_1,r}\to f_r(e^{i\theta_1}, .)$ 
uniformly on $\overline{\mathbb{D}}$ as $n\rightarrow \infty$. 
As we have computed above in \eqref{Eqn1}, one can show that  
$a_{mj}(P_{n,\theta_1,r})=0$ for all $m\in \mathbb{Z},n\in \mathbb{N}$, 
and $0<r<1$.  So by taking limits as $n\rightarrow \infty$ we have 
$a_{mj}(f_r)=0$ for all $0<r<1$. Finally taking the limit as $r\to 1^-$ 
we conclude that $a_{mj}(f)=0$ for $j\leq -1$.  The proof for $m\leq -1$ is similar. 
Hence we have shown that  $a_{mj}(f)=0$ for $j\leq -1$ or $m\leq -1$.  

 We define 
\[G_N(z,w)=\sum_{0\leq m,j\leq N}\left(1-\frac{m}{N+1}\right)
\left(1-\frac{j}{N+1}\right)a_{mj}(f)z^{m}w^{j}.\] 
Since we have shown $G_N\equiv F_N$ on $\mathbb{T}^2$ , we have 
$\|G_N-f\|_{L^{\infty}(\mathbb{T}^2)}\rightarrow 0$ 
as $N\rightarrow \infty$. Since $(G_N-f)(e^{i\theta},w)$ 
is holomorphic in $w$ and $(G_N-f)(z,e^{i\theta})$ is holomorphic 
in $z$, using the Maximum Modulus Principle for holomorphic 
functions, we have  
\[\|G_N-f\|_{L^{\infty}(b\mathbb{D}^2)}\leq \|G_N-f\|_{L^{\infty}(\mathbb{T}^2)}.\]  
So $\|G_N-f\|_{L^{\infty}(b\Omega)}\rightarrow 0$ as $N\rightarrow\infty$. 
Then Lemma  \ref{Lem7} implies that $\|H_{G_N-f}\|_e\to 0$ as $N\to 
\infty$. Furthermore, $\|H_f\|_e=\|H_{G_n-f}\|_e$ as $H_{G_N}=0$.  
Therefore, we conclude that $\|H_f\|_e=0$. That is, $H_f$ is 
compact on  $A^2(\mathbb{D}^2)$.
\end{proof}

\begin{remark}
	Even though we stated the previous proposition on $\mathbb{D}^2$ the same proof, 
	with trivial modifications, works on products of two discs. 
\end{remark}

Now we are ready for the proof of the main result.  

\begin{proof}[Proof of Theorem \ref{Thm1}]
	First we will prove the sufficiency. Assume that 
	$H_{\phi}$ is compact on $A^2(\D)$. If  there is no non-trivial analytic disc in the 
	boundary of $\D$ then there is nothing to prove. So assume that 
	$\Delta =f(\mathbb{D})$ is a non-trivial disc in $b\D$ such that 
	$\phi\circ f$ is not holomorphic. Without loss of generality we may assume that   
	$\Delta$ is horizontal as the proof for vertical discs is similar. Let 	us fix 
	$(z_0,w_0)\in \Delta$ 	and define $\alpha_j=(j-1)/j$. Then one can check that 
	$\|(w-w_0)^{-\alpha_j}\|_{L^2(H)}\to \infty$ as $j\to \infty$.  Let us define 
	\[g_j(w)=\frac{a_j}{(w-w_0)^{\alpha_j}}\]
	 where $a_j=1/\|(w-w_0)^{-\alpha_j}\|_{L^2(H_{\D})}$. Then $\|g_j\|_{L^2(H_{\D})}=1$ 
	 for all $j$. Furthermore,  $g_j\rightarrow 0$ uniformly on any compact subset in $\D$ 
	 as $j\rightarrow \infty$.  
Without loss of generality, we assume that $\Delta$ is the largest horizontal disc in 
$b\D$ passing through $(z_0,w_0)$ and $\phi_0$ be a continuous function on $\C^2$ 
such that $\phi_0(z,w)=\phi(z,w_0)$ for all $(z,w)\in \D$. That is, $\phi_0$ is the 
extension of  $\phi|_{\Delta}$ to $\mathbb{C}$ in $z$.  
Since $\phi_0$ is not holomorphic (as a function of $z$) on $\Delta$ we have 
$H^{\Delta}_{\phi_0}(1)\neq 0$. That is, $\|H^{\Delta}_{\phi_0}(1)\|_{L^2(\Delta)}>0$.  
Then by Corollary \ref{Cor1}, there exists $\beta>0$ and $\delta>0$ such that if 
$w\in H_{\D}$ and $|w-w_0|<\delta$, then 
\[\|H^{\Delta_{w}}_{\phi_0}(1)\|_{L^2(\Delta_{w})}>\beta.\]  
Let us define  $K=\{w\in H_{\D}:|w-w_0|\leq \delta\}$.  Then
\begin{align*}
\int_{H_{\D}}|g_j(w)|^2\int_{\Delta_w}|H^{\Delta_w}_{\phi_0}(1)(z)|^2dV(z)dV(w)
&\geq \int_{K}|g_j(w)|^2\int_{\Delta_w}|H^{\Delta_w}_{\phi_0}(1)(z)|^2dV(z)dV(w)\\
&\geq \beta^2\|g_j\|^2_{L^2(K)}.
\end{align*}
However, since $\|g_j\|^2_{L^2(H_{\D})}=1$ for all $j$ and $g_j\to 0$ uniformly 
on any compact set away from $w_0$ we conclude that $\|g_j\|^2_{L^2(K)}\geq 1/2$ 
for large $j$. Therefore, for large $j$ we have 
\[\int_{H_{\D}}|g_j(w)|^2\int_{\Delta_w}|H^{\Delta_w}_{\phi_0}(1)(z)|^2dV(z)dV(w)
\geq \frac{\beta^2}{2}>0.\]  
Then Lemma \ref{Lem4}  and Lemma \ref{Lem5}  
imply that $\| H_{\phi_0}g_j\|^2_{L^2(\Omega)}$ does not converge to $0$ 
as $j\rightarrow \infty$.  This contradicts Lemma \ref{Lem3} as we have assumed that 
$H_{\phi}$ is compact.

Finally we will prove the necessity. We assume $\phi\in C(\overline{\Omega})$ is such 
that $\phi\circ f$ is holomorphic for any holomorphic function 
$f:\mathbb{D}\rightarrow b\Omega$. Furthermore, we assume that $\D$ is 
not the product of two discs as that case is covered in 
Proposition \ref{PropBidisc}.   
Lemma \ref{Lem6} implies that there exists a family of functions 
$\{\psi_n\}\subset C^{\infty}(\overline{\Omega})$ such that 
\begin{itemize}
\item[i.] $\psi_n\circ f$ is holomorphic for any $n$ and any holomorphic $f:\mathbb{D}\to b\D$,
\item[ii.] $\psi_n\rightarrow \phi$ uniformly on $\Gamma_{\D}$ as $n\to \infty$.
\end{itemize}
Lemma \ref{Lem8} implies that $H_{\psi_n}$ is compact and Lemma \ref{Lem7} implies that 
$\|H_{\phi-\psi_n}\|_e\rightarrow 0$ as $n\rightarrow \infty$. Therefore, 
\[\|H_{\phi}\|_e=\|H_{\phi}\|_e-\|H_{\psi_n}\|_e\leq \|H_{\phi-\psi_n}\|_e.\]  
This implies $\|H_{\phi}\|_e=0$, proving that $H_{\phi}$ is compact on $A^2(\Omega)$. 
\end{proof}

\section*{Acknowledgment}
We would like to thank Trieu Le and Yunus Zeytunucu for valuable comments 
on a preliminary version of this manuscript.

\end{document}